\documentclass[a4paper,11pt]{article}
\usepackage[utf8]{inputenc}
\usepackage[sumlimits,intlimits]{amsmath}
\usepackage{mathrsfs,amssymb,amsthm,mathtools,bbm}
\usepackage[T1]{fontenc}
\usepackage{lmodern}
\usepackage{hyperref}

\usepackage{geometry}
\newtheorem{theorem}{Theorem}
\newtheorem{lemma}{Lemma}
\theoremstyle{definition}
\makeatletter
\newtheorem{przykl@d}{Example}
\newenvironment{przyklad}{%
  \begin{przykl@d}}{\qed\end{przykl@d}%
}
\makeatother
\theoremstyle{remark}

\newcommand{\bigO}{\mathit{O}}
\newcommand{\order}[1]{\bigO(n^{#1})}
\newcommand{\db}[3]{{#1}^{(#2)}_{#3}}
\newcommand{\transf}[3]{\db{#1}{#3}{#2}}
\newcommand{\met}[1]{$\mathcal{#1}$}
\newcommand{\metS}[1]{$\mathcal{S}_{#1}$}
\newcommand{\fN}{\mathbbm{N}}
\newcommand{\nsum}{\sum_{n=0}^\infty}

\title{Convergence acceleration of alternating series}

\author{%
Rafa\l{} Nowak%
\footnote{Institute of Computer Science, University of Wrocław,
  ul.~Joliot-Curie~15, 50-383 Wrocław, Poland,
  e-mail:
  \href{mailto:rafal.nowak@cs.uni.wroc.pl}{rafal.nowak@cs.uni.wroc.pl}}%
}
\date{\today}

\begin{document}
\maketitle

\begin{abstract}
We propose a~new simple convergence acceleration method for wide range class of convergent alternating series. It has some common features with Smith's and Ford's modification of Levin's and Weniger's sequence transformations, but its computational and memory cost is lower.
We compare all three methods and give some common theoretical results. Numerical examples confirm a~similar performance of all of them.
\end{abstract}

\section{Introduction}
This paper concerns the convergence acceleration of a~certain wide range class
of convergent alternating series.
More precisely:
$1^\circ$ a~new convergence acceleration method is given
and its certain theoretical properties are proved;
$2^\circ$ analogous properties for the Smith's and Ford's
\cite{SmithFord79} modification of Levin's and Weniger's $t$-transformations
(see also \cite[Eq.~(7.3-9)]{Weniger} or \cite[\S\,2.7]{BrezinskiZaglia91})
are proved and the similarities, as well as differences between all three
methods, are analyzed.

It is convenient to write the alternating series $\nsum a_n$ in the form
\begin{equation}\label{Ealtser}
\nsum (-1)^n\alpha_n,
\end{equation}
where $\alpha_n$ are all positive or negative.
In this paper we use the notation
\[ s_n \coloneqq \sum_{k=0}^{n} (-1)^k \alpha_k \]
for the partial sums of the series~\eqref{Ealtser}, whose limit,
in case of convergence, we denote by~$s$.

The mentioned class of alternating series concerns
convergent series~\eqref{Ealtser} such that $\alpha_n$ has an~asymptotic
expansion (as $n\to\infty$) of the form
\begin{equation}\label{Ealpha}
\alpha_n \sim x^n n^v\sum_{j=0}^\infty g_j n^{-j/r},
\end{equation}
provided that $x\in(0,1]$, $r\in\fN$ and $g_0\ne 0$.

It should be remarked that the series~\eqref{Ealtser}
with $\alpha_n$ satisfying relation~\eqref{Ealpha} is convergent, only
if we make a~certain additional assumption on numbers $v$ and~$x$.
Otherwise, we may deal with divergent series, whose summation may also be
useful. The detailed analysis of the convergence (and its acceleration) of the considered class of series (and more general, too) can be found in Sidi's book \cite[\S 8, \S 9]{Sidi03}. Namely, the class of~sequences $\{(-1)^n \alpha_n\}$
with $\alpha_n$ satisfying~\eqref{Ealpha} is a~subset of more general
class $\tilde{\mathbf{b}}^{(r)}$, given in~\cite[\S 6.6]{Sidi03}.

In the sequel the quantities
$
\beta_n\coloneqq\frac{\alpha_{n+1}}{\alpha_n}
$
play a~very important role.
One can verify that the asymptotic expansion in~\eqref{Ealpha} implies
\begin{equation}\label{Ebeta}
\beta_n\sim x\Bigl[1+\sum_{j=r}^\infty h_j n^{-j/r}\Bigr]
\qquad (n\to\infty);
\end{equation}
see, e.g., \cite[Thm.~6.6.4, p.~142]{Sidi03}.

From~\eqref{Ebeta} we conclude that
\( \beta_n = x\left[ 1+\order{-1} \right], \)
and thus
\begin{equation}\label{Ebetarel}
\frac{\beta_{n+1}}{\beta_{n+k}}=1+\order{-2}
\end{equation}
for any $k>0$.

In the simplest case of eq.~\eqref{Ealpha}
$v$ is~a~natural number and $r=1$.
This happens if $\alpha_n=x^nP(n)/Q(n)$
($P,\,Q$ being polynomials in~$n$).
Moreover, the coefficients of the polynomials $P,Q$ can depend on~$x$,
as in the following example:
\[\alpha_n=\frac{x^n(n+x^2)}{n+\ln(x+2)+1}.\]
A~much wider class of the series with $r=1$
refers to the hypergeometric functions.
Indeed, condition~\eqref{Ealpha} holds if the series~\eqref{Ealtser}
is identical, up to the~constant factor, with the function
\[{}_{p+1}F_p(a_1,a_2,\ldots,a_{p+1};b_1,b_2,\ldots,b_p;-x) = \sum_{n=0}^{\infty} (-1)^n \frac{(a_1)_n\cdots(a_{p+1})_n\,x^n}{(b_1)_n\cdots(b_p)_n\, n!},\]
which parameters $a_1,\ldots,a_{p+1}$, $b_1,\ldots,b_p$ and $x$ guarantee its alternation; notation $(z)_n$ means the Pochhammer symbol befined by $(z)_0 \coloneqq 1$, $(z)_j \coloneqq z(z+1)(z+2)\cdots(z+j-1)$, $(j\geq 1)$.
The relation~\eqref{Ebeta} is then evidently satisfied.

Further, the condition~\eqref{Ealpha} also holds if the terms $\alpha_n$
involve the roots in $n$ like, for~e.g., $\alpha_n=x^n\sqrt{n^2+2}$.
More examples with $r>1$ (and $v=-1$) are, for instance:
\begin{align}
\alpha_n & =\frac{x^n}{n+\sqrt{n}+1} \qquad (r=2),             \label{Eexr2} \\
\alpha_n & =\frac{x^n}{n+\sqrt{n}+\sqrt[3]{n+1}}, \qquad (r=6) \label{Eexr6}.
\end{align}
Let us note that such and similar terms $\alpha_n$
can be decomposed to a~sum of several terms $\db\alpha jn$,
for which the related quantities
$\db\beta jn\coloneqq \db\alpha j{n+1}/\db\alpha jn$
satisfy the equation \eqref{Ebeta} with $r=1$.
Indeed, one can decompose the expression in~\eqref{Eexr2} as follows:
\[\alpha_n=\frac{x^n(n-\sqrt{n}+1)}{(n+1)^2-n}=
\frac{x^n(n+1)}{n^2+n+1}-\frac{x^n\sqrt{n}}{n^2+n+1},\]
and thus the series $\nsum(-1)^n\alpha_n$ can be transformed
to the sum \(\nsum(-1)^n\transf\alpha n1+\nsum(-1)^n\transf\alpha n2,\)
where both quantities $\db\beta jn$, $j=1,2$,
satisfy the relation~\eqref{Ebeta} with $r=1$.
However, since all the summation methods, considered here, can be applied to
the~series~\eqref{Ealtser} satisfying~\eqref{Ebeta} with any natural number
$r$, it is hard to say if using these methods for each series
$\sum (-1)^n\db\alpha jn$ separately, gives actually better results.
One can check this is not true in the case of~\eqref{Eexr6}; see
Example~\ref{Ex7}.

The remainder of this paper is organized as follows.
Section~\ref{S:LW} deals with a~certain classic convergence
acceleration methods, such as Aitken's $\Delta^2$ method
and both Levin and Weniger transformations;
see~\cite{Aitken26}, \cite{Levin73} and~\cite{Weniger92}.
We consider there a~certain choice of the remainder estimates,
proposed by Smith and Ford in \cite{SmithFord79}
(see also \cite[\S\,2.7]{BrezinskiZaglia91}), in the case of Levin's
and Weniger's method, which we denote by the symbols \met L and \met W,
respectively.

It should be remarked that the $d^{(m)}$ transformation of Levin and Sidi
\cite{LevinSidi81} (with $m=r$) should also be an effective accelerator for
the considered series (see, e.g., \cite[\S6]{Sidi03}), as~well as more general
$\tilde d^{(m)}$ transformation developed by Sidi
(see the book \cite[pp.~147--148]{Sidi03} and the recent report~\cite{Sidi17}).

A~new method of convergence acceleration (denoted here by the symbol \met{S})
is presented in Section~\ref{S:Method}, which is followed in
Section~\ref{S:Wyniki} by a~discussion about common theoretical properties
including convergence acceleration theorem for all three methods
\met L, \met W and \met S.
In Section~\ref{S:Examples} we give some examples examining
the efficiency of the new method compared to the methods \met L and \met W.
All the examples except the last two consider the convergent series
\eqref{Ealtser} with $\alpha_n$ satisfying the relation \eqref{Ealpha}
with $r=1$.
One can check that the transformation $d^{(1)}$ of Levin and Sidi,
in the case of these examples, is equivalent to the method \met L,
provided the choice of parameters $R_l=l+1$, which is quite reasonable
for all of the considered examples in this paper.
Last two examples are the case with $r>1$ and thus, besides the comparison of
the efficiency of the methods \met S, \met L and \met W, we present the results
obtained by the $\tilde d^{(m)}$ transformation of Sidi (with $m=r$), as well.

Finally, in Section~\ref{S:divergent} we discuss the further properties
of the method~\met S, such as application to the summation of
divergent alternating series. Some remarks on efficient implementation of
the method~\met S are given therein, too.

\section{Levin and Weniger transformations}\label{S:LW}
The well-known Aitken's $\Delta^2$ method transforms a~given sequence $\{s_n\}$
into a~new sequence $\{s_n'\}$, defined by the formula
\begin{equation}\label{EAitk}
s'_n\coloneqq\frac{s_ns_{n+2}-s_{n+1}^2}{s_{n+2}-2s_{n+1}+s_n}.
\end{equation}
If the elements $s_n$ of the sequence to be transformed
are partial sums of alternating series~\eqref{Ealtser},
then
\begin{equation}\label{EAitks}
s'_n=\frac{\alpha_{n+2}s_{n}+\alpha_{n+1}s_{n+1}}{\alpha_{n+2}+\alpha_{n+1}}.
\end{equation}
Thus, the new sequence element $s'_n$ is a~weighted average of
the elements $s_n$ and $s_{n+1}$. These
weights are positive. Therefore, the numerical realization
of the Aitken's transformation has good stability properties.

It is important to note that the transformation~\eqref{EAitk}
can be easily iterated. Namely, one can use the sequence $\{s_n'\}$
as~a~sequence to be transformed, and obtain a~new sequence $\{s_n''\}$,
and so on; see, e.g., \cite[Eq.~(5.1-15)]{Weniger}. 
However, if the~elements $s_n$ are the partial sums
of~series~\eqref{Ealtser}, the process of iterating of the
transformation~\eqref{EAitks} is more subtle.
Indeed, in order to transform the sequence $\{s_n'\}$, one should replace
$\alpha_n$ in~\eqref{EAitks} with the terms of the series
\(s'_0+\nsum (s'_{n+1}-s'_n).\)
Computing these terms is not recommendable since one may be facing with
a~loss of significance caused by the cancellation of terms.
All the methods studied in this paper
do not have this disadvantage, although they are somehow derived from Aitken's
transformation.

The idea of the \textit{Levin transformation} \cite{Levin73} of the series $\nsum a_n$
is based on the assumption that the remainders of the partial sums have
the following Poincar\'e-type asymptotic expansion:
\begin{equation}\label{ELev}
s-s_m\sim\omega_m\sum_{j=0}^\infty \frac{d_j}{(m+b)^j} \qquad (\text{as }m\to\infty),
\end{equation}
where the shift parameter $b>0$ and~\textit{remainder estimates} $\omega_m$
should be chosen suitably for the considered class of the series.
Using the same notation as in~\cite{Weniger}, Levin transformation
can be expresses as follows:
\[ \db skn = %
\dfrac{\Delta^k\left[(n+b)^{k-1}\,\dfrac{s_n}{\omega_n}\right]}%
      {\Delta^k\left[(n+b)^{k-1}\,\dfrac{1}{\omega_n}\right]} = %
\dfrac{\displaystyle%
\sum_{j=0}^k (-1)^j \binom{k}{j} (n+j+b)^{k-1}\,\dfrac{s_{n+j}}{\omega_{n+j}}}%
{\displaystyle%
\sum_{j=0}^k (-1)^j \binom{k}{j} (n+j+b)^{k-1}\,\dfrac{1}{\omega_{n+j}}}.
\]
The choice of the remainder estimates has been widely discussed in the
literature; see, e.g., \cite{HomeierWeniger95}, \cite[\S\,7.3]{Weniger},
\cite[\S\,2.7]{BrezinskiZaglia91} or~\cite[\S\,5.3]{NumRec}.
However, the parameter $b$ is~usually chosen to be $1$.
In a~recent paper by~Abdalkhani and Levin~\cite{AbdalkhaniLevin2015}
the optimal value of this parameter was discussed
for a~certain variant of Levin transformation.

In the case of considered alternating series $\nsum a_n$,
the remainder has the following asymptotic expansion
\begin{equation}\label{Eremasympt}
  s-s_m \sim a_{m} \sum_{j=0}^\infty \gamma_j m^{-j/r}\qquad
   \text{as }m\to\infty,\quad \gamma_0 \neq 0;
\end{equation}
see, e.g., \cite[Thm.~6.6.6, pp.~145-147]{Sidi03}.
Thus it is recommendable to use
$\omega_m \coloneqq a_{m+1}$, i.e., Ford's and~Smith's
\cite{SmithFord79} modification of~Levin's $t$-transformation;
see also~\cite[\S\,2.7]{BrezinskiZaglia91}. In the sequel we denote
this method by the symbol \met{L}.

Any variant of Levin's method transforms the~sequence $\{s_n\}$
into a~doubly indexed sequence $\{\db skn\}$.
By definition, the element $\db skn$ is an~approximation
of the limit $s$ resulting from the system of the~equations
for $m=n,n+1,\ldots,n+k$, where only the terms with $j<k$ are retained.
Hence, the element $\db skn$ depends on all the values
$\alpha_j$ with $j\leq k+n+1$. For instance, in the case of method~\met L,
the element $\db s1n$ satisfies
the following system of two equations:
\[\db s1n-s_n=a_{n+1}d_0, \quad \db s1n-s_{n+1}=a_{n+2}d_0\]
with unknown $\db s1n$ and auxiliary coefficient $d_0$.
One can easily check that $\db s1n$ is exactly the value of $s_n'$
given by Aitken's transformation~\eqref{EAitks}.

\textit{Weniger transformation} is based upon an~assumption
similar to~\eqref{ELev}, and is given by
\begin{equation}\label{E:Weniger}
\db skn = %
\dfrac{\Delta^k\left[(n+b)_{k-1}\,\dfrac{s_n}{\omega_n}\right]}%
      {\Delta^k\left[(n+b)_{k-1}\,\dfrac{1}{\omega_n}\right]} = %
\dfrac{\displaystyle%
\sum_{j=0}^k (-1)^j \binom{k}{j} (n+j+b)_{k-1}\,\dfrac{s_{n+j}}{\omega_{n+j}}}%
{\displaystyle%
\sum_{j=0}^k (-1)^j \binom{k}{j} (n+j+b)_{k-1}\,\dfrac{1}{\omega_{n+j}}}.
\end{equation}
The only difference is that the powers $(m+b)^j$ are replaced by Pochhammer
symbols $(m+b)_j$; see \cite[\S\,8.2]{Weniger}.
Let us note that transformation~\eqref{E:Weniger} was invented independently
by Weniger and Sidi \cite{Sidi81} and later used by Shelef~\cite{Shelef} for
the numerical inversion of Laplace transforms.
Similarly to the method of Levin, we chose the remainder estimates $\omega_m = a_{m+1}$,
and denote this method by the symbol \met{W}.

One can check that both methods \met L and \met W
produce the double indexed arrays~of elements $\db skn$,
for which $\db s1n = s_n'$, i.e., both transformations $\db s1n$ are
equivalent to Aitken's transformation.
Further, both methods give the same values of $\db s2n$, which usually
are different than the ones obtained by the Aitken's iterated $\Delta^2$
process.

The parameter $b$ is usually chosen to be $1$ for both methods \met L and
\met W. We consider the same value for all presented numerical examples.

There are well-known recurrence formulas allowing for the efficient realization
of the Levin and Weniger transformations;
see, e.g., \cite[\S\,7.2, \S\,8.3]{Weniger}.
Both formulas are quite similar and use certain $3$-term recurrence relations
(see~\cite[Eqs. (7.3-2)--(7.2-6) and (8.3-1)--(8.3-5)]{Weniger})
satisfied by the following numerators $\db pkn$ and denominators~$\db qkn$:
\begin{equation}\label{ELWskn}
\db skn = \frac{\db pkn}{\db qkn} \qquad (k,\,n=0,1,\ldots).
\end{equation}

Their simplest variant may suffer from an~overflow that very often appears
during the recursive computation of numerators $\db pkn$ and
denominators~$\db qkn$. Hence,
it is recommendable to use so-called scaled versions of these
recurrence formulas; see~\cite[Eqs.~(7.2-8), (8.3-7)]{Weniger}.
For the case of $\{s_n\}$ being a~sequence of partial sums of the
alternating series~\eqref{Ealtser}, let us write these
$3$-term recurrence relations, for both methods \met L
and \met W, in the following way:
\begin{equation}\label{Epq0}
\db p0n\coloneqq s_n, \quad \db q0n\coloneqq 1,
\end{equation}
\begin{alignat}{2}
\label{Etphi}
\db{\tilde\varphi} kn &\coloneqq
\begin{dcases}
 \frac{(n+b+1)^{k-2}(n+k+b)}{(n+b)^{k-1}} & (\text{method $\mathcal{L}$}),\\
 \frac{n+b+2k-2}{n+b} & (\text{method $\mathcal{W}$}),
\end{dcases}
\end{alignat}
\begin{equation}\label{Erkn}
\db rkn\coloneqq\beta_{n+k}\db r{k-1}n+\db{\tilde\varphi}kn\db r{k-1}{n+1}
\qquad (r\equiv p,\,q;\; k\geq 1).
\end{equation}
Since the initial conditions~\eqref{Epq0} are the same for both methods,
which is not common in the literature,
the only difference comes from the choice of~the function $\db{\tilde\varphi}kn$
in the $3$-term recurrence relation~\eqref{Erkn}, satisfied by
numerators~$\db pkn$ and denominators~$\db qkn$.

For the convenience of later analysis and comparison with the new method~\met S,
let us observe that the quantities~$\db skn$, defined by~\eqref{ELWskn},
satisfy the recurrence relationship
\begin{equation}\label{ELWskn2}
  \transf snk =
  \frac{\beta_{n+1}}%
       {\beta_{n+1}+\transf{\varphi}nk}
  \transf sn{k-1}+
  \frac{\transf{\varphi}nk}%
       {\beta_{n+1}+\transf{\varphi}nk}
  \transf s{n+1}{k-1}
  \qquad (k>0),
\end{equation}
where
\begin{equation}\label{ELWphi}
  \transf{\varphi}nk = \transf{\tilde\varphi}nk
  \cdot\frac{\beta_{n+1}}{\beta_{n+k}}
  \cdot\frac{\transf q{n+1}{k-1}}{\transf qn{k-1}}.
\end{equation}
It is quite remarkable that the above formulas
permit to the compute the array $\db skn$ without
actually using the array of the numerators $\db pkn$.
Such realization of Levin and Weniger transformations
has probably not been considered in the literature, yet.

The following lemma displays some asymptotic property of the last fraction
in the right hand side of~equation~\eqref{ELWphi}, which we will use later
in the comparison involving all three methods~\met L, \met W, and \met S.
\begin{lemma}\label{Lq}
The quantities $\db qkn$ $(k>0)$, defined by eqs.~\eqref{Epq0}--\eqref{Erkn},
satisfy the relation
\[ \frac{\transf q{n+1}{k-1}}{\transf qn{k-1}} = 1 + \order{-2}. \]
\end{lemma}
\begin{proof}
Using induction on~$k$ and the relation~\eqref{Ebeta},
one can check that the quantities $\db qkn$
have the following formal power series expansion in variable $n^{-1/r}$:
\[ \db qkn \sim \db ik0 + \sum_{j=r}^\infty \db ikj n^{-j/r}. \]
From this, we conclude the result.
\end{proof}
It should also be remarked that, in case of alternating series
satisfying the relation~\eqref{Ebeta}, the functions $\db\varphi kn$,
defined by eq.~\eqref{ELWphi}, satisfy the relationship:
\begin{equation}\label{Ephirel2}
  \frac{\transf\varphi nk}{\transf{\tilde\varphi}nk} = 1 + \bigO(n^{-2}).
\end{equation}
It follows from eq.~\eqref{Ebetarel} and~Lemma~\ref{Lq}.

\section{Method \met{S}}\label{S:Method}
The starting point for the derivation of the aforementioned method~\met S
is Aitken's $\Delta^2$ sequence transformation, given by formula~\eqref{EAitks}.
However, the main idea is based upon the relationship involving the dependence
of the differences $\Delta s'_n$ on the terms $\alpha_n$, which allows us to
use, and also iterate, the formulas similar to~\eqref{EAitks}.
For instance, the simplest variant
(which we denote by symbol~\met S) produces the double indexed
array $\db skn$ of approximations~of the limit~$s$ of the
series~\eqref{Ealtser} by using the following recursive scheme:
\begin{align}
\db  s 0n & \coloneqq  s_n \qquad (n\geq 0), \notag \\
\db  s kn & \coloneqq
\frac{(n+1)\alpha_{n+2}\,\db s{k-1}n+(n+2k-1)\alpha_{n+1}\,\db s{k-1}{n+1}}
{(n+1)\alpha_{n+2}+(n+2k-1)\alpha_{n+1}} \label{Eskn}
\qquad  (k\geq 1,\; n\geq 0).
\end{align}
The above formula reduces for $k=1$ and
gives $\db s1n$ ($n>0$) identical with~$s'_{n}$ related to
Aitken's transformation~\eqref{EAitks}, and thus identical
with $\db s1n$ obtained by both methods \met L and \met W
(cf.~\eqref{ELWskn2}), as well.

According to~\eqref{Eskn} the element $\db skn$
is a~weighted average of~elements~$\db s{k-1}n$ and $\db s{k-1}{n+1}$.
We would like to note that this formula (for $k>1$)
can be derived by using the following brief analysis,
which we will discuss in more details in Section~\ref{S:Wyniki}.
Observing, at least experimentally, that for $k=1$ we
obtain the differences $\Delta\db s1n$ being
proportional to $n^{-2}\alpha_n$ (more precisely
$\Delta\db s1n/\alpha_n = \order{-2}$),
one can try to consider the change of the coefficients
in the~weighted average $\db s2n$ in order to obtain
$\Delta\db s2n \sim n^{-4}\alpha_n$, and so on;
here, and in the sequel, the forward difference operator $\Delta$ acts
upon the lower index $n$.
It is possible if one indeed replaces $\alpha_n$ in formula~\eqref{Eskn}
(having $k=1$) with $\alpha_n/(n+1)_2$,
which is exactly the formula~\eqref{Eskn}
with $k=2$. In general, for $k>1$, one should replace
$\alpha_n$ with $\alpha_n/(n+1)_{2k-2}$, which gives exactly
the formula~\eqref{Eskn}.

The following facts are evident or easy to check:
$1^\circ$ the recursive scheme, which defines the method \met{S},
differs from the ones for the~methods~\met{L} and~\met{W}.
The quantities $\db skn$ can be determined in a~straightforward way,
i.e., without computing their numerators and denominators (cf.~\eqref{ELWskn});
$2^\circ$ $\db  s kn$ is~a~function of the terms
$\alpha_{n}$, $\alpha_{n+1}$, \ldots, $\alpha_{n+k+1}$;
$3^\circ$ unlike the methods~\met L and \met W, the formula \eqref{Eskn},
which defines the method~\met S, is not a~consequence of any assumption
on the asymptotic behavior of the remainder estimates,
such as, e.g., \eqref{ELev}.
Formula \eqref{Eskn} is also not a~result of any~general expression for
the partial sums of the series, nor of the system of equations followed from
it. It is not even known if such an expression or a~system of equations exists;
$4^\circ$ in general, the quantities $\db skn$, $k>1$, are different from the
ones computed by the method~\met L or \met W (or Aitken's iterated $\Delta^2$
process). Indeed, the methods~\met L and~\met W give
\[
  \db{s}2n=\frac{%
      (n+1)(1+\beta_{n+1})\beta_{n+2}\db{s}1n+(n+3)(1+\beta_{n+2})\db{s}1{n+1}}
      {(n+1)(1+\beta_{n+1})\beta_{n+2}+(n+3)(1+\beta_{n+2})},
\]
which is the same as $\db s2n$ computed by the method \met S,
if the quantities $\beta_n$, related to the series~\eqref{Ealtser},
satisfy a~certain functional equation.

The justification of the efficiency of the method \met{S}
is discussed in Section~\ref{S:Wyniki}.
More precisely, it refers to a~more general method,
which we denote in the sequel by the symbol $\mathcal{S}_\varphi$,
given by the following formula:
\begin{equation}\label{Egenskn}
\db skn=\frac{\beta_{n+1}}{\beta_{n+1}+\db\varphi kn}
        \db s{k-1}n+\frac{\db\varphi kn}{\beta_{n+1}+\db\varphi kn}
        \db s{k-1}{n+1} \quad (k>0)
\end{equation}
(cf.~\eqref{ELWskn2}), where the arbitrary functions $\db\varphi kn$ are such
that
\begin{equation}\label{Ephi}
\db\varphi kn = 1+(2k-2)n^{-1} + \bigO(n^{-2}) \qquad(k>0)
\end{equation} and
$\beta_{n+1}+\db\varphi kn\ne 0$
(which usually follows from the former condition),
$n,k>0$.

The above conditions are satisfied if, for instance,
\begin{equation}\label{Epartphi}
\db\varphi kn=\db{\tilde\varphi}{k}n\qquad(k>0),
\end{equation}
where $\db{\tilde\varphi}kn$ are given by eq.~\eqref{Etphi}
related to the methods \met L and~\met W.
Moreover, if $\db{\tilde\varphi}kn$ corresponds to the method~\met W (with $b=1$),
the formula \eqref{Egenskn} is equivalent to~\eqref{Eskn}, and thus
the method $\mathcal{S}_\varphi$ becomes the~method~\met S.
On the other hand, one can also consider such functions $\db\varphi kn$, as
$[(n+3)/(n+1)]^{k-1}$ or $[(n+2)/(n+1)]^{2k-2}$, for which the condition~\eqref{Ephi}
can be easily checked.
Then, using eq.~\eqref{Egenskn} appears to be more costly,
but may have some advantages like better numerical stability.
We believe it is worth doing more analysis on this.
Let us remark that in both mentioned variants of the functions
$\db\varphi kn$, $\db\varphi 1n=1$ holds for all $n$, and thus $\db s 1n=s'_{n}$,
like in methods~\met L and \met W.

\section{General theoretical results}\label{S:Wyniki}
It is notable that the similarities between all three methods \met L, \met W
and \met S follow from eqs.~\eqref{ELWskn2} and~\eqref{Egenskn},
which vary depending on the choice of the functions $\db\varphi kn$;
cf.~\eqref{ELWphi},~\eqref{Ephi} and~\eqref{Epartphi}.
For instance, the difference between the choice involving the function~\eqref{ELWphi}
(which gives the methods~\met L, \met W) and \eqref{Epartphi} (method~\met S)
is well depicted by the relation~\eqref{Ephirel2}.

Let us note that the statement of Theorem 1 that follow,
in the case of alternating series~\eqref{Ealtser} with $\alpha_n$ satisfying
\eqref{Ealpha} with $r=1$, is very similar to the classic results for Levin's
and Weniger's transformations; see, e.g., Weniger's report \cite[\S 13]{Weniger}
For the detailed analysis of the convergence acceleration of the alternating
series~\eqref{Ealtser} with $\alpha_n$ satisfying \eqref{Ealpha} with $r=1$
and the application of Levin's transformations to it, we refer to papers
by Sidi \cite{Sidi79}, \cite{Sidi80}.

For our consideration, the relation~\eqref{Ephirel2} plays the main role
in deriving a~theoretical properties common for the new method~\metS{\varphi}
and both methods \met L and \met W.
For the sake of analysis of all three methods, let us use common symbol $\db skn$
to denote the elements of the array computed by them.
It is important that in all three cases the quantities $\db skn$
satisfy the 3-term recurrence relation~\eqref{Egenskn}
where the~functions $\db\varphi kn$ depend on the considered
method; cf.~\eqref{ELWskn2}.

In order to study the convergence acceleration performed by the mentioned
methods, it is recommendable to investigate the differences $\Delta\db skn$.
Indeed, the quantities $\Delta\db skn$ (together with the element $\db{s} k0$),
$k>0$, are the terms of the series resulting
from the corresponding sequence transformation.
The efficiency of the method depends on whether these series
(for consecutive $k$) converge to limit $s$ faster and faster.
Hence, it is reasonable to compare the differences $\Delta\db skn$
to the original terms $\alpha_{n+1}$. For this reason, let us define
the following quantities:
\begin{equation}\label{EDkn}
  \db Dkn\coloneqq \frac{(-1)^{n+k+1}\Delta\db skn}{\alpha_{n+1}}, \qquad
  \db Bkn\coloneqq \frac{\beta_{n+1}}{\beta_{n+1}+\db\varphi kn}.
\end{equation}

\begin{lemma}\label{LDB}
The quantities $\db Dkn, \db Bkn$ satisfy the following relationship:
\[\db Dkn=\beta_{n+1}(1-\db Bk{n+1})\db D{k-1}{n+1}-\db Bkn\db D{k-1}n
  \qquad (n\geq 0, \; k>0).\]
\end{lemma}
\begin{proof} It follows from eq.~\eqref{Egenskn} that
\[
\db{s} kn =\db{s} {k-1}{n+1}-\db Bkn\Delta\db{s} {k-1}n,
\]
and thus
\[\Delta\db{s} kn=(1-\db Bk{n+1})\Delta\db{s}{k-1}{n+1}+
\db Bkn\Delta\db{s}{k-1}n.\]
Now, by multiplying both hand sides by $(-1)^{n+k+1}/\alpha_{n+1}$,
we obtain the result.
\end{proof}

As mentioned in the previous section, the quantities $\db s2n$ can be identical
for all three methods if $\beta_n$ satisfies a~certain functional equation.
Indeed, for any series~\eqref{Ealtser}, the quantities $\db s2n$, defined by
eq.~\eqref{Egenskn}, are the same as the ones for the methods \met L and \met W,
if one takes the following function $\db\varphi2n$ written in terms
of $\beta_n$:
\[
  \db\varphi 2n=
  \frac{(n+3)(1+\beta_{n+2})\beta_{n+1}}{(n+1)(1+\beta_{n+1})\beta_{n+2}}.
\]
For $k>2$, the analysis of such similarities seems to be meaningless.
However, it is quite notable that for the choice
of the functions $\db\varphi kn$, such as~\eqref{ELWphi}
and~\eqref{Etphi}, equation~\eqref{Egenskn}, which defines the
method~\metS\varphi, leads to the method~\met L or~\met W.
Undoubtedly much more important is the meaning of the condition~\eqref{Ephi}
involving the functions $\db\varphi kn$, which are, for all three methods,
such that
\begin{equation}\label{Ephiprop}
\db\varphi{k}{n} \left[1-(2k-2)n^{-1}\right] = 1+\bigO(n^{-2});
\end{equation}
cf.~\eqref{Ephi}, \eqref{Epartphi} and~\eqref{Ephirel2}.
Namely, it is summarized in the following theoretical results.

\begin{theorem}\label{TDser}
Let $\db skn$ be the two-dimensional array computed by the method
\metS{\varphi}, given in~\eqref{Egenskn}, applied to~the series~\eqref{Ealtser}
with~$\alpha_n$ satisfying~\eqref{Ealpha}.
Then the differences $\Delta\db{s} kn$ satisfy the following relation:
\begin{equation}\label{EDa}
\frac{(-1)^{n+k+1}\Delta\db{s} kn}{\alpha_{n+1}}\sim
n^{-2k}\,\sum_{j=0}^\infty \db dkj n^{-j/r} \quad \text{as} \quad n\to\infty.
\end{equation}
\end{theorem}
\begin{proof} The proof follows by induction on $k$.
Since $\Delta\db{s} 0n=\Delta s_n=(-1)^{n+1}\alpha_{n+1}$,
the series in~\eqref{EDa} (for $k=0$) simplifies to the constant $\db d00=1$.
Now, let $k>0$ be given. 
Taking into account the~relation~\eqref{Ebeta} in the definition of the
quantities~$\db Bkn$, we conclude that
\begin{equation}\label{EBkn}
\db Bkn=\xi+\db gk1 n^{-1}+\db gk2 n^{-(r+1)/r}\ldots,\;\;
\text{where}\;\xi\coloneqq \frac{x}{x+1},
\end{equation}
and thus
\begin{equation}\label{EBkn1}
\db Bk{n+1}=\xi+\db gk1(n+1)^{-1}+\db gk2(n+1)^{-(r+1)/r}+\ldots=
\db Bkn+\bigO(n^{-2}).
\end{equation}
In the same way, one may check that
\( \db\varphi{k}{n+1} = \db\varphi{k}{n} + \bigO(n^{-2}). \)
Moreover, from~\eqref{Ebetarel}, it immediately follows that
\( {\beta_{n+1}}/{\beta_{n+2}} = 1 + \bigO(n^{-2}). \)
Hence, we get
\begin{equation}\label{Ebetaphi}
\beta_{n+1}(1-\db Bk{n+1}) = \beta_{n+1}
\frac{\db\varphi{k}{n+1}}{\beta_{n+2}+\db\varphi{k}{n+1}}
=\frac{\beta_{n+1}}{\beta_{n+2}} \db\varphi{k}{n+1} \db Bk{n+1}
=\db\varphi{k}{n}\db Bkn+\bigO(n^{-2}).
\end{equation}
From the principle of the induction, it follows that
\begin{align*}
\db D{k-1}n & =\db d{k-1}0 n^{-2k+2}+\db d{k-1}1 n^{-2k+2-1/r}+\ldots, \\
\db D{k-1}{n+1} & =\db d{k-1}0 (n+1)^{-2k+2}+\db d{k-1}1 (n+2)^{-2k+2-1/r}+\ldots= \\
& =[1-(2k-2)n^{-1}+\bigO(n^{-2})]\db D{k-1}n.
\end{align*}
Therefore, in view of~\eqref{Ephiprop}, we conclude that
\begin{align*}
\db Dkn &= \bigl(\db\varphi{k}{n}\db Bkn+\bigO(n^{-2}) \bigr) \db D{k-1}{n+1}-\db Bkn\db D{k-1}n\\
&= \db\varphi{k}{n}\db Bkn\bigl(1-(2k-2)n^{-1}+\bigO(n^{-2})\bigr)\db D{k-1}n-\db Bkn\db D{k-1}n\\
&=\bigl((1+\bigO(n^{-2}))\db Bkn-\db Bkn+\bigO(n^{-2})\bigr)\db D{k-1}n
=\bigO(n^{-2})\db D{k-1}n,
\end{align*}
and the proof is complete.
\end{proof}
The evident meaning of the above result is as follows: the larger is the value
of~$k$, the less are the absolute values of the differences $\Delta\db{s} kn$
(at least for sufficiently great values of~$n$),
and thus, the faster is the convergence of $\db{s} kn$ to $s$.
Similar results, but only for the methods~\met L and~\met W (and with $r=1$
in \eqref{Ealpha}), can be found in the Weniger's report \cite[Thms.~13-5, 13-9, pp. 114, 117]{Weniger}.

It is also worth considering the influence of the choice of the
functions $\db\varphi kn$ on the asymptotic behavior of the differences
$\Delta\db skn$ that appear in Theorem~\ref{TDser}.
Of course, this dependence is related to the values $\db dkj$, which,
in general, are usually unknown. This is somewhat displayed in the
following result.

\begin{theorem}\label{TDD}
Let the quantities $\db skn$ be as in the previous theorem.
Then the following relation links the quantities $\db\varphi kn$ with $\db Dkn$,
given in~\eqref{Ephi} and \eqref{EDkn}, respectively:
\[\frac{\db Dkn}{\db Dk{n+1}}=\db\varphi{k+1}{n}[1+\bigO(n^{-2})].\]
\end{theorem}

\begin{proof} By replacing $k$ with $k+1$ in Lemma~\ref{LDB},
we have that
\[\beta_{n+1}(1-\db B{k+1}{n+1})\frac{\db Dk{n+1}}{\db Dkn}=
\db B{k+1}n+\frac{\db D{k+1}n}{\db Dkn}.\]
Hence, by Theorem~\ref{TDser}, the quotient $\frac{\db D{k+1}n}{\db Dkn}$ is of
order $\bigO(n^{-2})$,
and thus, replacing $k$ with $k+1$ in~\eqref{Ebetaphi} yields
\[\frac{\db Dkn}{\db Dk{n+1}}=
\frac{\beta_{n+1}(1-\db B{k+1}{n+1})}{\db B{k+1}n+\bigO(n^{-2})}=
\frac{\db\varphi{k+1}{n}\db B{k+1}n+\bigO(n^{-2})}{\db B{k+1}n+\bigO(n^{-2})}.\]
Now, the result follows from~\eqref{EBkn}.
\end{proof}

As we mentioned in the previous section,
for each method \met L, \met W and \met S,
we have $\db\varphi 1n=1$. The final remark is that $\Delta\db s1n/\alpha_{n+1}$
simplifies to
\begin{equation}\label{Dsi1}
\frac{\Delta\db{s} 1n}{\alpha_{n+1}}=
(-1)^{n}\beta_{n+1}\Delta\left(\frac{1}{1+\beta_{n+1}}\right),
\end{equation}
which for many series can be easily expressed just in terms of $n$.

\section{Numerical examples}\label{S:Examples}
Let us consider the method \met{S}, defined by eq.~\eqref{Eskn},
and the mentioned variants of Levin and Weniger transformations,
defined by formulas \eqref{Epq0}--\eqref{Erkn}
and denoted by symbols \met L and~\met W, respectively.

If the terms $\alpha_n$ of the series~\eqref{Ealtser} to be transformed
are sufficiently simple, and if $k$ is rather small,
one can try to find explicit expression
for the quantities $\db Dkn$ and verify the statements of
Theorems \ref{TDser} and~\ref{TDD}, given in the previous section.
Let us recall that for $k=0,1$ all three methods produce the same
values of $\db skn$, while, for $k=2$, it is evidently true only for the
methods of Levin and Weniger.

For instance, if $\alpha_n=1/(n+1)$, then we have
\begin{equation*}
  \db D1n  =-\frac{n+2}{(n+3)(2n+5)(2n+7)}=-\frac{1}{4}\,n^{-2}+\bigO(n^{-3}), \\
\end{equation*}
\begin{equation*}
\db D2n =
\begin{dcases}
-{\frac {n+2}{ ( n+3 )  ( 2{n}^{2}+14n+25) ( 2{n}^{2}+10n+13 ) }},\\
-{\frac { ( n+2 )  ( 10{n}^{3}+91{n}^{2}+273n+264 ) }{ ( 2n+9 )  ( 2{n}^{2}+13n+22 )  ( 2n+5 )  ( 2{n}^{2}+9n+11 ) ( 2n+7 )  ( n+3 ) }}
\end{dcases}
\end{equation*}
(the first formula corresponds to the methods \met L and \met W;
second --- to the method \met S).
This is in agreement with Theorem~\ref{TDser}, since
\begin{equation*}
\db D2n =
\begin{dcases}
-\frac{1}{4}n^{-4}+\bigO(n^{-5})  & \text{(methods \met L, \met W)},  \\
-\frac{5}{16}n^{-4}+\bigO(n^{-5}) & \text{(method \met S)}.
\end{dcases}
\end{equation*}
The comparison of the leading coefficients of the asymptotic expansions of
the values $\db D2n$ shows that the methods~\met L and \met W
yield a~little bit better result than the method \met S.
In contrast, the method \met S is better than the others for $k=3$.
Indeed, one may check that
\begin{equation*}
\db D3n =
\begin{dcases}
-\frac{3}{16}n^{-6}+\bigO(n^{-7})  & \text{(methods \met L, \met W)},  \\
-\frac{9}{64}n^{-6}+\bigO(n^{-7}) & \text{(method \met S)}.
\end{dcases}
\end{equation*}
Further comparison, i.e., for $k>3$, seems to be pointless.

For
\[\alpha_n\coloneqq \frac{(2n+2)!}{4^nn!(n+2)!}\,x^n\]
the expression for $\db D1n$ is rather complicated, i.e.,
\[\db D1n=-\frac{(n+3)(2n+5)(n^3+3n-1)x^2}{W(n)W(n+1)}=-\frac{x^2}{2(x+1)^2}\,n^{-2}+\bigO(n^{-3}),\]
where $W(n)\coloneqq (n+3)(2n+5)x+2(n+2)(n+4)$.
However, one can check that
\[\frac{1}{\db\varphi 2{n+1}}\,\frac{\db D1n}{\db D1{n+1}}=
1-\frac{17+15x}{2+2x}\,n^{-2}+\bigO(n^{-3}),\]
which is in agreement with Theorem~\ref{TDD}.

We compared the performance of the methods \met{L}, \met{W} and~\met{S}
numerically, by applying them to several alternating series~\eqref{Ealtser}
of different types. For each example below, we present as follows:
the form of the series (including the values $x,v,r$ in the
relation~\eqref{Ealpha}) and its limit $s$,
the accuracy of the quantities $\db sk0$, $k=3,4,\ldots$,
for all three methods (first row corresponds to the method \met S,
second row --- \met L, third row --- \met W).
First five examples are the case of $r=1$ and last two are not.

Let us remark that $\tilde d^{(m)}$ transformation of Sidi
(with $m=r$ and $\hat\sigma=0$)
is also effective accelerator of the considered series; see,~e.g., \cite[\S6.6.4, pp. 147--148]{Sidi03}.
The accuracy of the quantities $\db{\tilde d}{m,k}{0}$ for $k=3,4,\ldots$ is
given (in the fourth row) in the case of last two examples, since only
then ${\tilde d}^{(m)}$ transformation is not equivalent to the method \met L.
We choose $\hat\sigma=0$ since
for the the considered series $\nsum a_n$ we have
\[ a_n / \Delta a_n \in \tilde{\textbf{A}}_0^{(\sigma,m)}
  \quad\text{with}\quad\sigma=0;\]
see \cite[Thm.~6.6.5]{Sidi03} and \cite[\S6.6.1]{Sidi03} for the details on
the class $\tilde{\textbf{A}}_0^{(\sigma,m)}$.

Here, the accuracy of the approximation~$\tilde{s}$ of the sum~$s\neq 0$
is measured by $-\log_{10}|\tilde{s}/s-1|$, i.e., by
the number of exact significant decimal digits.
As it was mentioned before, the classic methods of Levin and Weniger give the
same values of $\db skn$ for $k=0,1,2$.

All the numerical experiments were made
using \textsf{IEEE 754} double extended precision, i.e.~$80$-bit floating-point arithmetic,
which means about 19 decimal digits precision.

Let use note the in the examples that follow, it appears that the numerical results
produced by the method~\met{S} seem to be similar to those obtained by
the classic Levin and Weniger transformations, as well as the Sidi's generalization of them.

\begin{przyklad}[$x=1, v=-2, r=1$]
\[\nsum \frac{(-1)^n}{n^2+1}=0.63601\,45274\,91066\,581\]
\footnotesize
\[\begin{array}[15]{rrrrrrrrrrrrrrr}
  3.6 &  5.1 &  6.1 &  7.0 &  8.7 &  9.2 & 10.4 & 11.6 & 12.5 & 14.1 & 14.6 & 17.1 & 16.8 \\
  3.9 &  5.6 &  6.1 &  7.4 &  9.1 &  9.7 & 11.1 & 12.8 & 13.4 & 14.8 & 16.6 & 17.2 & 18.7 \\
  5.1 &  5.1 &  6.2 &  7.6 &  9.2 & 10.2 & 10.9 & 11.7 & 12.7 & 13.6 & 14.5 & 15.5 & 16.4
\end{array}\]
\end{przyklad}


\begin{przyklad}[$x=1, v=-3/2, r=1$] 
\begin{align*}
{}_3F_2\bigl(1,1,\tfrac 32;2,2;-1\bigr) & =
\nsum \frac{(-1)^n(2n+1)!}{4^n[(n+1)!]^2}= \\
& =4\ln\bigl(0.5+\sqrt{0.5}\bigr)=0.75290\,56258\,83839\,086
\end{align*}
{\footnotesize
\[\begin{array}[15]{rrrrrrrrrrrrrr}
  4.0 &  5.2 &  7.1 &  7.7 &  8.9 & 10.4 & 11.1 & 13.1 & 13.4 & 14.8 & 15.7 & 16.9 & 18.0 \\
  4.1 &  5.4 &  7.0 &  8.1 &  9.1 & 10.4 & 13.9 & 12.9 & 14.0 & 15.8 & 16.7 & 17.7 &  19.0 \\
  4.7 &  5.9 &  7.5 & 10.2 & 10.2 & 11.2 & 12.3 & 13.4 & 14.5 & 15.5 & 16.6 & 17.6 & 18.7
\end{array}\]
}
\end{przyklad}

\begin{przyklad}[$x=1, v=-1, r=1$] 
\[\nsum \frac{(-1)^{n}}{n+1}=\ln 2=0.69314\,71805\,59945\,309\]
{\footnotesize
\[\begin{array}[14]{rrrrrrrrrrrrrr}
  3.9 &  5.3 &  7.0 &  7.6 &  9.5 & 10.0 & 11.4 & 12.4 & 13.5 & 14.7 & 15.7 & 17.1 & 17.9 \\
  4.0 &  5.3 &  7.0 &  8.1 &  9.1 & 10.5 & 12.3 & 12.8 & 14.1 & 17.0 & 16.6 & 17.8 & 18.8 \\
  4.9 &  5.9 &  7.2 &  8.6 & 10.1 & 11.6 & 13.1 & 14.6 & 16.1 & 17.6 & 19.0 &  19.0 & 19.0
\end{array}\]
}
\end{przyklad}

\begin{przyklad}[$x=2/3, v=-1/2, r=1$] 
\[{}_2F_1\bigl(\tfrac 32,2;3;-\tfrac 23\bigr)=
\nsum \frac{(-1)^n(2n+2)!}{6^nn!(n+2)!}=0.59032\,00617\,95601\,049\]
{\footnotesize
\[\begin{array}[13]{rrrrrrrrrrrrr}
  4.2 &  5.3 &  6.6 &  8.3 &  9.7 & 10.8 & 13.2 & 13.5 & 15.0 & 16.3 & 17.4 & 18.3 & 18.4 \\
  4.1 &  5.1 &  6.2 &  7.5 &  8.8 & 10.2 & 11.8 & 13.9 & 15.0 & 16.2 & 17.9 & 19.0 & 18.6 \\
  4.3 &  6.4 &  7.8 &  9.4 & 11.1 & 12.8 & 14.6 & 16.4 & 18.0 & 19.0 & 18.8 & 18.8 & 18.6
\end{array}\]
}
\end{przyklad}

\begin{przyklad}[$x=1/2, v=1/2, r=1$]
\[\nsum \bigl(-\tfrac 12\bigr)^n\sqrt{n+1}=
0.56602\,56214\,93012\,046\]
{\footnotesize
\[\begin{array}[14]{rrrrrrrrrrrrrr}
  4.7 &  7.2 &  7.8 &  9.9 & 10.8 & 12.6 & 13.8 & 15.4 & 16.9 & 18.5 & 18.1 & 18.1 & 18.1 \\
  4.4 &  5.7 &  7.2 &  8.8 & 10.5 & 12.4 & 14.9 & 17.6 & 17.6 & 18.1 & 18.1 & 18.1 & 18.1 \\
  5.5 &  7.0 &  8.4 &  9.8 & 11.1 & 12.3 & 13.5 & 14.6 & 15.8 & 17.0 & 19.0 & 18.1 & 18.1
\end{array}\]
}
\end{przyklad}

\begin{przyklad}[$x=1, v=-1, r=6$] \label{Ex7}%
\begin{equation}\label{Eex7}\nsum \frac{(-1)^n}{n+\sqrt n+\sqrt[3]{n+1}} =
0.81139\,68270\,43132\,432\end{equation}
{\footnotesize
\[\begin{array}[14]{rrrrrrrrrrrrrr}
  4.4 &  5.9 &  7.2 &  7.8 &  8.6 &  9.4 & 10.3 & 11.1 & 11.9 & 12.7 & 13.5 & 14.2 & 15.0 \\
  4.5 &  5.8 &  6.8 &  7.6 &  8.3 &  9.1 &  9.9 & 10.6 & 11.3 & 12.1 & 12.8 & 13.5 & 14.1 \\
  5.2 &  6.3 &  7.8 &  8.7 & 10.0 & 10.6 & 12.5 & 12.4 & 14.4 & 14.2 & 15.8 & 16.0 & 17.3 \\
  4.5 &  5.6 &  7.1 &  7.6 &  9.0 &  9.9 & 10.8 & 12.4 & 13.0 & 14.0 & 16.4 & 16.1 & 17.2
\end{array}\]
}%
It should be remarked that, although the quantities
$\beta_n=\alpha_{n+1}/\alpha_n$ for the considered series satisfy
the relation~\eqref{Ebeta} with $r=6$, it is possible to
find the decomposition
$\alpha_n = \sum_{i=1}^{6} (-1)^{\eta_i}\db{\alpha}in,$
such that each $\db \beta in\coloneqq \db\alpha i{n+1}/\db\alpha i{n}$
satisfies~\eqref{Ebeta} with $r=1$. However, we checked that
applying the method~\met S to all of the decomposed series
$\nsum (-1)^n \db\alpha in$ separately does not give better results.
\end{przyklad}

\begin{przyklad}[$x=1, v=e-7/2, r=2$] %
\[\nsum \frac{(-1)^n n^e}{n+(n+1)^{7/2}}=-0.02049\,06107\,716 \]
{\footnotesize
\[\begin{array}[14]{rrrrrrrrrrrrrr}
  2.2 &  3.0 &  4.0 &  5.1 &  6.1 &  7.2 &  8.4 &  9.6 & 10.7 & 11.9 & 13.4 & 14.5 & 15.8 \\
  2.3 &  3.3 &  4.6 &  6.6 &  6.8 &  8.0 &  9.7 & 10.0 & 10.9 & 12.0 & 13.0 & 13.8 & 14.7 \\
  3.1 &  3.6 &  4.4 &  5.2 &  6.1 &  7.0 &  7.9 &  8.7 &  9.5 & 10.4 & 11.2 & 12.0 & 12.9 \\
  2.3 &  3.8 &  4.5 &  5.7 &  6.4 &  8.1 &  8.5 & 10.4 & 10.7 & 12.8 & 12.9 & 15.3 & 15.1
\end{array}\]
}%
\end{przyklad}

\section{Some additional remarks}\label{S:divergent}
It should be noted that assumption~\eqref{Ealpha} does not imply the convergence
of~the series~\eqref{Ealtser}. It ensures only the correctness of all the
theorems given in Section~\ref{S:Wyniki}.
However, if the~series is~divergent, the method~\met{S} can still be used.
Such a~summation of divergent series can also be performed by the classic
transformations of~Levin and Weniger, which is well discussed in the mentioned
report of Weniger~\cite{Weniger}.

For instance, the power series $1-x+x^2-x^3+\ldots$ has the limit $(1+x)^{-1}$
if $|x|<1$, and diverges if $x=1$ (here, one can consider the so-called
\textit{antilimit}~$1/2$). It is easy to check that the method~\met S gives
$\db s1n=1/2$ for all $n\geq 0$.

The power series $1-2x+3x^2-4x^3+\ldots$,
which converges to $(1+x)^{-2}$ if $|x|<1$,
is a~more complicated example. For~$x=1$, one should expect that
the method~\met{S} will give the approximations of the number $1/4$.
Indeed, although the consecutive approximations $\db s 0n=s_n$ are equal
to $1,-1,2,-2,3,-3,\ldots$, one obtains the following approximations
$\db skn$ (for $k=1,2,3$), converging to $1/4$:
\begin{align*}
  \db s 1n & =\frac{1}{4}\left[1+\frac{(-1)^{n+1}}{(2n+5)}\right], \\
  \db s 2n & =\frac{1}{4}\left[1-\frac{3\,(-1)^{n+1}}{(2n+3)(2n+5)(2n+7)}\right], \\
  \db s 3n & =\frac{1}{4}\left[1-\frac{3(n-3)(-1)^{n+1}}{(2n+3)(2n+5)(2n+7)(2n+9)(2n^2+11n+13)}\right].
\end{align*}
This should be compared with the following results
obtained using the method~\met L:
\begin{align*}
  \db s 2n & =\frac{1}{4}\left[1-\frac{(-1)^{n+1}}{2n^3+16n^2+40n+31}\right], \\
  \db s 3n & =\frac{1}{4}\left[1-\frac{3\,(-1)^{n+1}}{4n^5+62n^4+372n^3+1084n^2+1544n+867}\right].
\end{align*}
It should be also remarked that the method \met W gives significantly better results,
since $\db s3n=1/4$ for all $n\geq 0$.

It is notable that in a~more general case, namely for the hypergeometric series
\[(1+x)^{-\rho}={}_1F_0(\rho,-x)=\nsum \frac{(\rho)_n}{n!}\,(-x)^n\]
($\rho \in \mathbbm R\setminus\{-1\}$), we are dealing with the convergence for $x=1$
if and only if $\rho<0$. However, the~method~\met S transforms in the
first step (which is equivalent to Aitken and the first step of
Levin and Weniger transformations), the above series into the
hypergeometric series
\[\frac{1}{\rho+1}\,{}_2F_1\left(\rho,\tfrac12(\rho-1);\tfrac12(\rho+3);-1\right),\]
which converges if $\rho<2$; one can prove this by using
formula~\eqref{Dsi1}. The analysis of the next steps of the
method~\met S seems to be quite difficult,
but undoubtedly some of them give the series converging to~$2^{-\rho}$.
From this and the previous examples, it appears that the results
produced by the method~\met{S} seem to be similar to those
obtained by the classic Levin and Weniger transformations,
in the case of summation of divergent series, as well.

Finally, it should also be remarked that method~\met S can be programmed quite
efficiently. For that purpose, it is recommendable to write the recurrence
scheme~\eqref{Eskn} in the following way:
\[\db  s kn=\frac{\db s {k-1}n+\db tkn\db s {k-1}{n+1}}{1+\db tkn}, \]
where $\db tkn\coloneqq \frac{n+2k-1}{(n+1)\beta_{n+1}}.$
Then the computation of the value $\db s kn$ costs
$2$ divisions, $2$ multiplications and $2$ additions,
where at least, each one of them has only one floating point argument;
this calculation does not include the cost of computing the numbers $\beta_n$,
since these values are being used many times and for many series they are
often much more simple than the terms $\alpha_n$.
On the other hand, the methods~\met{L} and~\met{W}, in their simplest variants,
can be programmed such that computation of $\db skn$ costs $1$ division,
$4$ multiplications and $2$ additions. However, for the sake of numerical
stability, it is recommendable to use a~certain scaled version
of $3$-term recurrence formulas defining these methods
(see~\cite[Eqs.~(7.2-8), (8.3-7)]{Weniger}), which indeed significantly
increases their complexity.

Again, let us remark that programming the method \met S (using the above formula)
does not involve the computation of separate two-dimensional arrays of numerators
and denominators as it is in the case of Levin and Weniger transformations
(and Sidi's generalizations, as well).
Thus the new method is significantly cheaper if we take into account the memory usage.

\section*{Acknowledgements}
I~would like to express my sincere gratitude to Prof.~S.~Paszkowski,
who initiated the proposed convergence acceleration method.
The provided assistance and valuable comments were crucial for this research.

Finally, I would like to thank the both reviewers. Their comments significantly
improved the presentation of this paper.

\bibliographystyle{abbrv}
\bibliography{bibliography}

\begin{thebibliography}{10}

\bibitem{AbdalkhaniLevin2015}
J.~Abdalkhani and D.~Levin.
\newblock On the choice of $\beta$ in the $u$-transformation for convergence
  acceleration.
\newblock {\em Numer. Algorithms}, 70(1):205--213, 2015.

\bibitem{Aitken26}
A.~C. Aitken.
\newblock {On Bernoulli's numerical solution of algebraic equations}.
\newblock {\em Proceedings Royal Soc. Edinburgh}, 46:289--305, 1926.

\bibitem{BrezinskiZaglia91}
C.~Brezinski and M.~{Redivo Zaglia}.
\newblock {\em Extrapolation Methods: Theory and Practice}, volume~2 of {\em
  Studies in Computational Mathematics}.
\newblock North-Holland, 1991.

\bibitem{HomeierWeniger95}
H.~H.~H. {Homeier} and E.~J. {Weniger}.
\newblock On remainder estimates for {L}evin-type sequence transformations.
\newblock {\em Comput. Phys. Commun.}, 92(1):1--10, 1995.

\bibitem{Levin73}
D.~Levin.
\newblock Development of non-linear transformations for improving convergence
  of sequences.
\newblock {\em {Int. J. Comput. Math.}}, 3:371--388, 1973.

\bibitem{LevinSidi81}
D.~{Levin} and A.~{Sidi}.
\newblock {Two new classes of nonlinear transformations for accelerating the
  convergence of infinite integrals and series.}
\newblock {\em {Appl. Math. Comput.}}, 9:175--215, 1981.

\bibitem{NumRec}
W.~H. Press, S.~A. Teukolsky, W.~T. Vetterling, and B.~P. Flannery.
\newblock {\em Numerical Recipes 3rd Edition: The Art of Scientific Computing}.
\newblock Cambridge University Press, New York, NY, USA, 2007.

\bibitem{Shelef}
R.~Shelef.
\newblock New numerical quadrature formulas for laplace transform inversion by
  bromwich's integral (in hebrew).
\newblock Master's thesis, Technion-Israel Institute of Technology, Haifa,
  1987.

\bibitem{Sidi79}
A.~{Sidi}.
\newblock {Convergence properties of some nonlinear sequence transformations.}
\newblock {\em {Math. Comput.}}, 33:315--326, 1979.

\bibitem{Sidi80}
A.~{Sidi}.
\newblock {Analysis of convergence of the T-transformation for power series.}
\newblock {\em {Math. Comput.}}, 35:833--850, 1980.

\bibitem{Sidi81}
A.~{Sidi}.
\newblock {A new method for deriving Pad\'e approximants for some
  hypergeometric functions.}
\newblock {\em {J. Comput. Appl. Math.}}, 7:37--40, 1981.

\bibitem{Sidi03}
A.~Sidi.
\newblock {\em Practical Extrapolation Methods - Theory and Applications},
  volume~10 of {\em Cambridge monographs on applied and computational
  mathematics}.
\newblock Cambridge University Press, 2003.

\bibitem{Sidi17}
A.~Sidi.
\newblock Acceleration of convergence of some infinite sequences $\{A_n\}$
  whose asymptotic expansions involve fractional powers of $n$.
\newblock Technical report, Dept. of Computer Science, Technion-Israel
  Institute of Technology, 2017.
\newblock \url{https://arxiv.org/abs/1703.06495}.

\bibitem{SmithFord79}
D.~A. {Smith} and W.~F. {Ford}.
\newblock {Acceleration of linear and logarithmic convergence.}
\newblock {\em {SIAM J. Numer. Anal.}}, 16:223--240, 1979.

\bibitem{Weniger}
E.~J. Weniger.
\newblock Nonlinear sequence transformations for the acceleration of
  convergence and the summation of divergent series.
\newblock {\em Comput. Phys. Rep.}, 10:189--371, 1989.

\bibitem{Weniger92}
E.~J. Weniger.
\newblock {Interpolation between sequence transformations}.
\newblock {\em Numer. Algorithms}, 3(1--4):477--486, 1992.

\end{thebibliography}
\end{document}